\documentclass[12pt]{amsart}
\usepackage{bbm}
\usepackage{graphicx}
\usepackage{mathrsfs}
\usepackage{amsmath, amssymb, amsthm, latexsym}
\usepackage{amssymb,amscd,amsmath}
\usepackage{latexsym}
\usepackage{MnSymbol}
\usepackage{enumitem}
\usepackage[center]{caption}
\usepackage{tikz}
\newcounter{braid}
\newcounter{strands}

\DeclareMathAlphabet{\bsf}{OT1}{cmss}{bx}{n}

\pgfkeyssetvalue{/tikz/braid height}{1cm}
\pgfkeyssetvalue{/tikz/braid width}{1cm}
\pgfkeyssetvalue{/tikz/braid start}{(0,0)}
\pgfkeyssetvalue{/tikz/braid colour}{black}
\pgfkeys{/tikz/strands/.code={\setcounter{strands}{#1}}}

\makeatletter
\def\cross{%
  \@ifnextchar^{\message{Got sup}\cross@sup}{\cross@sub}}

\def\cross@sup^#1_#2{\render@cross{#2}{#1}}

\def\cross@sub_#1{\@ifnextchar^{\cross@@sub{#1}}{\render@cross{#1}{1}}}

\def\cross@@sub#1^#2{\render@cross{#1}{#2}}

\def\render@cross#1#2{
  \def\strand{#1}
  \def\crossing{#2}
  \pgfmathsetmacro{\cross@y}{-\value{braid}*\braid@h}
  \pgfmathtruncatemacro{\nextstrand}{#1+1}
  \foreach \thread in {1,...,\value{strands}}
  {
    \pgfmathsetmacro{\strand@x}{\thread * \braid@w}
    \ifnum\thread=\strand
    \pgfmathsetmacro{\over@x}{\strand * \braid@w + .5*(1 - \crossing) * \braid@w}
    \pgfmathsetmacro{\under@x}{\strand * \braid@w + .5*(1 + \crossing) * \braid@w}
    \draw[braid] \pgfkeysvalueof{/tikz/braid start} +(\under@x pt,\cross@y pt) to[out=-90,in=90] +(\over@x pt,\cross@y pt -\braid@h);
    \draw[braid] \pgfkeysvalueof{/tikz/braid start} +(\over@x pt,\cross@y pt) to[out=-90,in=90] +(\under@x pt,\cross@y pt -\braid@h);
    \else
    \ifnum\thread=\nextstrand
    \else
     \draw[braid] \pgfkeysvalueof{/tikz/braid start} ++(\strand@x pt,\cross@y pt) -- ++(0,-\braid@h);
    \fi
   \fi
  }
  \stepcounter{braid}
}

\tikzset{braid/.style={double=\pgfkeysvalueof{/tikz/braid colour},double distance=1pt,line width=2pt,white}}

\newcommand{\braid}[2][]{%
  \begingroup
  \pgfkeys{/tikz/strands=2}
  \tikzset{#1}
  \pgfkeysgetvalue{/tikz/braid width}{\braid@w}
  \pgfkeysgetvalue{/tikz/braid height}{\braid@h}
  \setcounter{braid}{0}
  \let\sigma=\cross
  #2
  \endgroup
}
\makeatother

\input xypic
\newtheorem{theorem}{Theorem}
\newtheorem{proposition}[theorem]{Proposition}

\newtheorem{lemma}[theorem]{Lemma}

\newtheorem{corollary}[theorem]{Corollary}


\def\Z{\mathbb{Z}}

\def\Q{\mathbb{Q}}

\def\R{\mathbb{R}}

\def\N{\mathbb{N}}

\def\qed{\hfill$\square$\medskip}

\def\Zpk{\mathbb{Z}/p^{k}}
\def\Zpk1{\mathbb{Z}/p^{k-1}}

\newcommand{\rref}[1]{(\ref{#1})}

\newcommand{\beg}[2]{\begin{equation}\label{#1}#2\end{equation}}
\def\r{\rightarrow}

\def\sl2{\widetilde{SL_{2}(\Z)}}

\title[Adjunctions in equivariant homotopy theory]{On some adjunctions in equivariant stable homotopy theory}
\author{Po Hu, Igor Kriz and Petr Somberg}
\thanks{The authors acknowledge support by grant GA\,CR P201/12/G028.
Kriz also acknowledges the support of a Simons Collaboration Grant.}

\begin{document}
\maketitle

\begin{abstract}
We investigate certain adjunctions in derived categories of equivariant spectra, including a right adjoint to
fixed points, a right adjoint to pullback by an isometry of universes, and a chain of two right adjoints to
geometric fixed points. This leads to a variety of interesting other adjunctions, including a chain of $6$ (sometimes
$7$) adjoints involving the restriction functor to a subgroup of a finite group on equivariant spectra indexed
over the trivial universe.
\end{abstract}

\section{Introduction}

In equivariant stable homotopy theory, we study, for a compact Lie group $G$, generalized cohomology
theories which are stable under suspensions by $1$-point compactifications of
finite-dimensional $G$-representations. Such theories are represented in the derived category
$DG$-$\mathscr{U}$-spectra where $\mathscr{U}$ is the complete universe (\cite{lms}), i.e.
a countably-dimensional $G$-inner product space containing infinitely many copies of all irreducible $G$-representations.
Certain functors come up naturally when studying these theories, for example the {\em fixed point functor}
$(?)^G$, which is used in calculating homotopy groups, and also the {\em geometric fixed point functor}
$\Phi^G$, which was quite important in the work of Hu and Kriz \cite{real} on Real-oriented $\Z/2$-spectra,
as well as, later, in the solution by Hill, Hopkins and Ravenel \cite{hhr} of the Kervaire invariant 1 problem.

\vspace{3mm}
The left derived functor of the functor $\Phi^G$ has a right adjoint, which again has a right adjoint.
(Throughout this paper, we will 
focus on the derived context, so this language will often be omitted.) In a recent paper \cite{sanders}, Balmer, Dell'Ambrogio and Sanders 
investigated a general framework in which certain ``geometric functors" between tensor triangulated categories
have two right adjoints. In fact, they proved that under suitable assumptions (see Section \ref{sres} below),
only three possibilities arise, namely a chain of 3, 5 or infinitely many adjoints on both sides. Geometric fixed points,
in the case of a complete universe, satisfy the assumptions of \cite{sanders}, and in this context (see also
\cite{bs}), it 
seemed interesting to look at this example more closely. 

\vspace{3mm}
In a recent paper \cite{sanders1}, B. Sanders investigates another example
of the 3-adjunction \cite{sanders}, namely ``inflation", i.e. the fixed $G$-spectrum, indexed over the complete
universe, associated with a spectrum $X$. In fact, Sanders introduces a beautiful formalism which enables an abstract
treatment of the Adams isomorphism. The right adjoint of the inflation functor on the level of derived categories is
the fixed point functor $(?)^G$ of a spectrum indexed over the complete universe. Again, by the work
of \cite{sanders}, this functor has an additional right adjoint on derived categories, which the authors
of the present paper also observed independently in connection with their work on spectral Lie algebras
\cite{hks}. Unlike the case of geometric fixed points, this functor, however,
is much harder to describe, and even now remains somewhat mysterious.

\vspace{3mm}
Inspecting this example more closely suggests looking beyond the case of a complete
universe. The reason is that the fixed point functor $(?)^G$ on the derived category $DG$-$\mathscr{U}$-spectra
really is a composition of two functors, the first one of which is pullback $i^*$ via the inclusion
$i:\mathscr{U}^G\r \mathscr{U}$ from the ``trivial universe" $\mathscr{U}^G\cong \R^\infty$. On spectra
indexed over
the trivial universe (which represent generalized cohomology theories only stable with respect
to ordinary suspensions), geometric and ordinary fixed points are the same thing.

\vspace{3mm}
A natural question then arises: Does the pullback $i^*$ with respect to an isometry of universes also
have a right adjoint on the level of derived categories? Are the observations of the previous paragraphs
also true for non-complete universes? The answer to the first question is yes, as is, for the most part, the answer
to the second question. It is important to note, however, that we are now leaving the world of the assumptions
of \cite{sanders}, since for spectra indexed over a non-complete universe, the important assumption of
\cite{sanders} that compact objects be strongly dualizable precisely fails for those triangulated categories. 

\vspace{3mm}
Since inflation is a case of functoriality with respect to change of groups (the case of a surjection), what about
restriction, i.e. the case of an injection of groups? In this case, for complete universes, we have the
well known Wirthm\"{u}ller isomorphism \cite{lms}, which also was a part of the inspiration for \cite{sanders}
as well as, for example, Fausk, Hu and May \cite{fhm}. What happens in the case of non-complete universes?
It turns out that in this case, which does not satisfy the assumptions of \cite{sanders},
we always have a chain of $4$ adjoints. However, in the case of a finite group and the trivial universe, 
we show that
there is, in fact, a chain of $6$ adjoints, and in the case of a finite abelian group and the trivial subgroup
for the trivial universe, there is a chain of $7$ adjoints. In special cases, this can be worked out
quite explicitly. We also have counterexamples
showing that in general, these chains of adjunctions extend no further. 

\vspace{3mm}
In a closely related case of the endofunctor of smashing with a finite spectrum, we again have a chain of
infinitely many adjunctions on both sides in the case of a complete universe. Generally, we always have a chain 
of three adjoints, and for the case of a finite group and the trivial universe, we have a chain of
$5$ adjoints ($6$ adjoints in the case of a primary cyclic group). In both cases, these chains 
extend no further in general. These endofunctors, in fact, give us, at least in principle, a description of the right adjoint to 
pullback along an isometry of universes.

\vspace{3mm}

The purpose of the present paper is to treat these situations as completely as we are, at the moment, able, 
both in terms of positive statements and counterexamples, since they
appear to be important for the foundations of equivariant stable homotopy theory.
Here is a more detailed description of the situations we consider:

\begin{enumerate}[label=(\Alph*)]

\item
\label{i1} {\bf Restriction.} It is a tradition from group cohomology to separate pullback with respect 
to a homomorphism of groups into the case when the homomorphism is injective (restriction) and surjective
(inflation). For equivariant spectra, too, the two cases behave somewhat differently, and for this reason, we, too, 
treat them separately.
The forgetful functor $res=res_{H}^G:DG$-$\mathscr{U}$-spectra$\r DH$-$\mathscr{U}$-spectra 
where $H\subseteq G$ is 
a closed subgroup of a compact Lie group $G$, and $\mathscr{U}$ is {\em any $G$-universe} (not necessarily complete)
is well known to have a left adjoint $G\ltimes_H?$ and a right adjoint $F_H[G,?)$. We show that
$F_H[G,?)$ also has a {\em right adjoint} $\Xi^G_H$. 
The ``left projection formula" in the sense of \cite{sanders}, (3.11) is well known
to hold, but we show that the ``right projection formula" in the sense of \cite{sanders}, (2.16) is false in general.
It is well known that if $\mathscr{U}$ is the complete $G$-universe, then the left and right adjoints to $res_{H}^G$ are 
``shifts" of each other, and hence the chain of adjunctions described extends to an infinite chain of adjunctions
on both sides. This is the Wirthm\"{u}ller isomorphism.  However, we show by example that for a general universe, 
$G\ltimes_H(?)$ may not have a left adjoint, and $\Xi^G_H$ may not have a right adjoint. Thus, we have a chain 
of 4 adjoints in general. We show however that in the case when $\mathscr{U}=\R^\infty$ is the ``trivial" universe
and $G$ is finite,
then $G\ltimes_H ?$ has two more adjoints to the left (thus giving a chain of a total of $6$ adjoints), and when $G$
is abelian and finite, then $G\ltimes_H ?$ has three adjoints to the left (thus giving a chain of $7$ adjoints). 
In both cases, we have examples showing that this chain of adjoints may not extend any further. 

\vspace{3mm}

\item
\label{i2} {\bf Smashing with a finite spectrum.}
Let $X$ be a retract of a finite cell $G$-$\mathscr{U}$-spectrum for any universe $\mathscr{U}$. Then the 
functor $X\wedge?:DG$-$\mathscr{U}$-spectra$\r DG$-$\mathscr{U}$-spectra has a right adjoint 
$F(X,?)$. We show that this functor has a further right adjoint $R(X,?)$. In the case when $\mathscr{U}$
is a complete universe, $F(X,?)=DX\wedge ?$, and hence $R(X,?)=X\wedge ?$. However, we show that 
in general, $R(X,?)$ does not have a right adjoint. On the other hand, we show that $X\wedge ?$ always has two
left adjoints (leading to a chain of at least $5$ adjoints), and for $G=\Z/p$ it has exactly three (leading to
a chain of 6 adjoints).

\vspace{3mm}

\item
\label{i3}{\bf Change of universe.}
For an isometry of $G$-universes $i:\mathscr{U}\r\mathscr{V}$ for a compact Lie group $G$, 
the universe change functor 
$i^*:DG$-$\mathscr{V}$-spectra$\r DG$-$\mathscr{U}$-spectra is well known to have a left adjoint,
which we denote by $i_\sharp$. We prove that it also has a right adjoint, which we denote by $i_*$. (In \cite{lms}, 
$i_\sharp$ was denoted by $i_*$. However, in all sorts of contexts of
sheaf theory, $i_*$ is always the right adjoint, which is why we use the alternate notation.)
We show that in general, $i_\sharp$ does not have a left adjoint, and $i_*$ does not have a right adjoint. The
right
projection formula \cite{sanders} (2.16) is false. We have a chain of $3$ adjoints in this case.

\vspace{3mm}
\item
\label{i4}
{\bf Inflation.}
For a compact Lie group $G$, a $G$-universe $\mathscr{U}$, and an onto homomorphism of compact
Lie groups $G\r J=G/H$ for a closed normal subgroup $H$ of $G$, we have the functor
$inf=inf_G^J={}_{\mathscr{U}}inf_G^J
:\mathscr{U}^H$-$J$-spectra$\r\mathscr{U}$-$G$-spectra. This functor is most universal
when $\mathscr{U}$ is an $H$-fixed universe,  since in general we have
\beg{einfch}{{}_\mathscr{U}inf_G^J=i_\sharp\circ{}_{\mathscr{U}^H}inf_G^J}
where $i:\mathscr{U}^H\r\mathscr{U}$ is the inclusion. In the case when $\mathscr{U}$ is an $H$-fixed universe,
$inf_G^J$ has a left adjoint $?/H$ and a right adjoint $(?)^H$ which has a right adjoint 
$\widetilde{E[H]}\wedge inf_G^J$, which has a further right adjoint $F(\widetilde{E[H]},?)^H$. No further
right or left adjoints exist, so this is a chain of $5$ adjoints. Further, the left and right projection formulas
\cite{sanders} (3.11), (2.16) hold and $\widetilde{E[H]}$ is the ``dualizing object" in the sense of 
\cite{sanders}, even though
the assumptions of \cite{sanders} are not satisfied. If $\mathscr{U}$ is an arbitrary universe, then
in general, ${}_{\mathscr{U}}inf_G^J$ has no left adjoint, and its right adjoint $(?)^H$ has a right adjoint 
$i_*(\widetilde{E[H]}\wedge inf_G^J)$, which in general has no right adjoint, so this is a chain of 3 adjoints.
The right projection formula is true. In the case when $\mathscr{U}$ is a complete universe, this case
satisfies the assumptions of \cite{sanders}.

\vspace{3mm}
\item
\label{i5}
{\bf Geometric fixed points.}
For a compact Lie group $G$ and any $G$-universe $\mathscr{U}$, and $J=G/H$ for a closed 
normal subgroup $H\subseteq G$,
we may consider the functor
$\Phi^H={}_\mathscr{U}\Phi^H
=(\widetilde{E[H]}\wedge ?)^H:DG$-$\mathscr{U}$-spectra$\r DJ$-$\mathscr{U}^H$-spectra. This functor
is most universal when $\mathscr{U}$ is the largest universe with given $\mathscr{U}^H$ (up to isomorphism), 
since for a general embedding
$i:\mathscr{U}\r\mathscr{V}$ where $\mathscr{V}^H\cong \mathscr{U}^H$, we have
$${}_\mathscr{U}\Phi^G={}_\mathscr{V}\Phi^G\circ i_\sharp.$$
In the case when $\mathscr{U}$ is a complete universe, the assumptions of \cite{sanders} 
are satisfied, and a $3$-duality therefore holds. In other words, ${}_\mathscr{U}\Phi^G$ has
a right adjoint, which again has a right adjoint, and the projection formula holds. This adjunction in general 
extends no further, so this is a case of $3$-duality in the sense of \cite{sanders}. If $\mathscr{U}$ is not a complete
universe, we still have a $3$-duality, and the projection formula still holds.

\end{enumerate}

\section{The main results}\label{sres}

Let us begin by reviewing the setup of Balmer, Dell'Ambrogio and Sanders \cite{sanders}. In the greatest generality, 
they talk about triangulated categories. A triangulated category $\mathscr{T}$ 
is called {\em compactly generated} if it has coproducts,
and has a set of compact objects $\mathscr{G}$ which generate $\mathscr{T}$. To generate means that if for
$x\in Obj(\mathscr{T})$, for every $z\in \mathscr{G}$, $\mathscr{T}(z,x)=0$, then $x=0$. An object $x$
of $\mathscr{T}$ is called
compact if $\mathscr{T}(x,?)$ sends coproducts in $\mathscr{T}$ to coproducts of abelian groups. 

\vspace{3mm}

In this paper, we consider the derived categories of $G$-$\mathscr{U}$-spectra where $G$ is a compact Lie group, and
$\mathscr{U}$ is a $G$-universe. These categories are compactly generated, where the generators are (integral) 
suspensions of suspension
spectra of orbits (by closed subgroups). These spectra generate essentially by definition of the derived category
(see \cite{lms, ekmm}). The fact that these generators are compact is widely known and widely used, but since we could
not locate a proof in the literature, we present one in the Appendix.

\vspace{3mm}

The authors of \cite{sanders} use the following two facts to construct adjoint functors:

\begin{lemma}\label{lex} (\cite{sanders}, Corollary 2.3)
Let $F:\mathscr{T}\r\mathscr{S}$ be an exact (=triangle-preserving) functor between triangulated categories, where $\mathscr{T}$ is compactly generated. Then 

(a) $F$ has a right adjoint if and only if it preserves coproducts

(b) $F$ has a left adjoint if and only if it preserves products.

\end{lemma}
\qed

\begin{lemma}\label{lex2} (\cite{sanders}, Lemma 2.5)
Let $F:\mathscr{T}\r \mathscr{S}$ be left adjoint to $G:\mathscr{S}\r\mathscr{T}$ where $F,G$ are exact functors
between triangulated categories, and $\mathscr{T}$ is compactly generated. Then $F$ preserves compact objects
if and only if $G$ preserves coproducts.
\end{lemma}
\qed

\vspace{3mm}

The authors of \cite{sanders} consider a functor $f^*:\mathscr{T}\r\mathscr{S}$
and investigate patterns of adjunction of the following form:
\beg{ethree}{\diagram
\dto^{f^*} &&\dto^{f^{(1)}}\\
&\uto_{f_*}&
\enddiagram}
and
\beg{efive}{\diagram
&\dto^{f^*} &&\dto^{f^{(1)}}&\\
\uto_{f_{(1)}}&&\uto_{f_*}&&\uto_{f_{(-1)}}
\enddiagram}
Their assumption is that $f^*:\mathscr{T}\r\mathscr{S}$ is an exact functor between compactly generated
tensor triangulated categories which preserves the
symmetric monoidal structure, and preserves coproducts. They additionally assume that both in $\mathscr{T}$
and $\mathscr{S}$, compact objects are strongly dualizable. Under these assumptions, they prove that  \rref{ethree}
always occurs, and additionally, one has the {\em right projection formula} stating that we have an isomorphism
\beg{erproj}{\diagram x\wedge f_*(y)\rto^(.45)\cong & f_*(f^*(x)\wedge y)\enddiagram
}
where \rref{erproj} is the canonical morphism. (We write the symmetric monoidal structure as $\wedge$, since 
in the topological contexts we discuss, it is always the smash product). The dualizing object
by definition is $f^{(1)}(\mathbbm{1})$ where $\mathbbm{1}$ is the unit of the
symmetric monoidal structure (in all our cases, this is the sphere spectrum $S$ in the appropriate
category). Then the authors of \cite{sanders} prove
that the additional right adjoint $f_{(-1)}$ exists if and only if the additional left adjoint $f_{(1)}$ exists, leading
to the \rref{efive} scenario. Additionally, if that happens, they prove the {\em left projection formula}
\beg{elproj}{\diagram f_{(1)}(f^*(x)\wedge y) \rto^(.55)\cong & x\wedge f_{(1)}(y)\enddiagram 
}
where \rref{elproj} is, again, the canonical morphism.

The authors of \cite{sanders} also prove that if either $f_{(1)}$ admits a left adjoint or $f_{(-1)}$ admits
a right adjoint, then \rref{efive} extends to an infinite chain of adjunctions on both sides.

\vspace{3mm}
The category $DG$-$\mathscr{U}$-spectra for a compact Lie group $G$ and a universe $\mathscr{U}$ is always
compactly generated, where the compact generators are (de)suspensions of suspension spectra of orbits. 
Further, compact objects are strongly dualizable when $\mathscr{U}$ is a complete universe (i.e. contains 
representatives of
all
isomorphism classes of
finite-dimensional $G$-representations). However, compact objects are not strongly dualizable in general, notably when
$G$ is non-trivial and $\mathscr{U}$ is the trivial universe, containing only copies of the trivial representation. 
Therefore, the conclusions of \cite{sanders} do not, strictly speaking, apply to most of our situations. 
Nevertheless, Lemmas \ref{lex} and \ref{lex2} have the immediate 

\begin{corollary}\label{cor1}
If a functor $f^*:\mathscr{T}\r\mathscr{S}$ is an exact functor
where $\mathscr{T},\mathscr{S}$ are compactly generated triangulated 
categories, preserves compact objects and coproducts, then it has a right adjoint $f_*$ which in turn has another 
right adjoint $f^{(1)}$, i.e. the scenario \rref{ethree} occurs.
\end{corollary}

\begin{proof}
By Lemma \ref{lex}, a right adjoint $f_*$ exists and, by Lemma \ref{lex2}, it preserves coproducts. Additionally,
since $\mathscr{T}$ is compactly generated, distinguished triangles can be tested by long exact sequences
on morphism groups from compact objects, so $f_*$ preserves distinguished triangles by adjunction and
by the fact that $f^*$ preserves compact objects. Therefore, the additional right adjoint $f^{(1)}$ exists by
Lemma \ref{lex2}.
\end{proof}

\vspace{3mm}
\noindent
{\bf Comment:} When a triangulated category is compactly generated, compact objects are precisely objects
of the smallest thick subcategory generated by the compact generators. Therefore, for the functor $f^*$ to
preserve compact objects, it is sufficient to show that it sends the compact generators to compact objects.

\vspace{3mm}
In the scenarios described in the introduction,  the functors $res_H^G$ of Case \ref{i1}, $X\wedge ?$ of Case \ref{i2},
$i_\sharp$ of Case \ref{i3}, $inf^{J}_{G}$ of Case \ref{i4} and ${}_{\mathscr{U}}\Phi^G$ of
Case \ref{i5} are all exact functors which
preserve compact generators and coproducts, so Corollary \ref{cor1} applies if we take these
functors for $f^*$. In other words, a right adjoint $f_*$ exists which, in turn, has again a right adjoint $f^{(1)}$.
Interestingly, the functors $f^{(1)}$ in this case do not appear to have been noticed in most cases. In this note,
we will consider some examples. 

\vspace{3mm}
Before that, however, let us discuss the \rref{efive} scenario, i.e. the case of a chain of $5$
adjoints. It turns out that under our weaker assumptions,
it is false that the existence of one of the
adjoint $f_{(-1)}$, $f_{(1)}$ would imply
the existence of the other. Nevertheless, the existence of the functors $f_{(1)}$ and $f_{(-1)}$ can still be tested using 
the following

\begin{corollary}
\label{cfive}
If $f^*:\mathscr{T}\r\mathscr{S}$ is an exact functor between compactly generated triangulated categories
which preserves coproducts and compact objects. Then the functor $f^{(1)}$ has a right
adjoint if and only if $f_*$ preserves compact objects, and $f^*$ has a left adjoint if and only if it preserves products.
\end{corollary}

\begin{proof}
The first statement follows from Lemma \ref{lex2} and Lemma \ref{lex} (a), and the second statement follows
from Lemma \ref{lex} (b).
\end{proof}

\vspace{3mm}
In case \ref{i1} and \ref{i2} of the Introduction, the functor $f_{(1)}$ exists, but in general the functor
$f_{(-1)}$ does not. In case \ref{i3}, in general neither $f_{(1)}$ nor $f_{(-1)}$ exists. In case \ref{i4} for a trivial
universe, both $f_{(1)}$ and $f_{(-1)}$ exist, and in cases \ref{i5} for the complete universe, neither $f_{(1)}$ nor
$f_{(-1)}$ exists in general.

\vspace{3mm}
Additionally, in cases \ref{i1}, \ref{i3}, \ref{i4} and \ref{i5}, the functor $f^*$ preserves symmetric 
monoidal structure, and hence we can ask about the projection formulas \rref{erproj} and \rref{elproj}. 
In case \ref{i1}, the formula \rref{elproj} holds, but the formula \rref{erproj} is false in general. 
In case \ref{i3}, the formula \rref{erproj} is false in general. In case \ref{i4} for a trivial universe, 
the formula \rref{elproj} is false in general, 
and the formula \rref{elproj} is true for the trivial universe but 
false in general. In case \ref{i5}, the formula \rref{erproj} for
a complete universe holds (this follows from Theorem 2.15 of \cite{sanders}, since the assumptions in this
case are satisfied). For a general universe, the formula \rref{erproj} is false in this case.

\vspace{3mm}
We now turn to discussing each case of the Introduction individually in more detail.

\vspace{3mm}
\subsection{Inflation for the case of an $H$-fixed universe}\label{inffixed}
We will discuss this part of Case \ref{i4} first, since it will be used in our other discussions. Here 
we have a normal subgroup $H$ of $G$, $J=G/H$, and $f^*=inf_G^J:DJ$-$\mathscr{U}$-spectra$
\r DG$-$\mathscr{U}$-spectra, where $\mathscr{U}=\mathscr{U}^H$
is an $H$-trivial $G$-universe.
Unless $H=\{e\}$, the assumptions of \cite{sanders} are never satisfied, because the suspension
spectrum of the orbit $G/\{e\}$
is not strongly dualizable in $DG$-$\mathscr{U}$-spectra.  Nevertheless, we have
the ``$5$-scenario" of \cite{sanders}. A left adjoint of $f^*$ is
$f_{(1)}=?/H$ (more precisely its left derived functor, i.e. $?$ should be a cell
spectrum). On the other hand, a right adjoint is $f_*=?^H$, which in turn has a right adjoint
$f^{(1)}=(inf_G^J(?))\wedge \widetilde{E[H]}$ where $E[H]=E\mathscr{F}[H]$ is the classifying space
of the family $\mathscr{F}[H]$ of closed subgroups of $G$ which $H$ is not subconjugate to. (Note that
since $H$ is normal, this is just the family of all subgroups not containing $H$.)
Here $\widetilde{X}$
means the unreduced suspension of a $G$-space $X$.
To see this, mapping a $G$-$\mathscr{U}$-cell spectrum $X$ to $inf_G^JY\wedge \widetilde{E[H]}$ where
$Y$ is a $J$-$\mathscr{U}$-cell spectrum, the cells of $X$ with isotropy superconjugate to $H$ (which 
cannot be attached
to cells with isotropy not superconjugate to $H$, since $\mathscr{U}$ is $H$-fixed) must map to $Y$, while there
is no obstruction to mapping any other cells to $inf_G^JY\wedge \widetilde{E[H]}$; a similar argument
applies to homotopies. The functor $f^{(1)}$ has a right adjoint $f_{(-1)}=F(\widetilde{E[H]},?)^H$.
It is also worth noting that not only the functor $f^*$, but also the functor $f_*$ is strongly symmetric
monoidal, while $f^{(1)}$ preserves the smash product but not the unit.

The left projection formula \rref{elproj} states that
\beg{elinfp}{\diagram (inf_G^J(X)\wedge Y)/H\rto^(.55)\sim & X\wedge (Y/H),\enddiagram}
which is true. In effect, both sides clearly preserve homotopy cofibers, so it suffices to consider the case
when $X=J/K_+$ for some closed subgroup $K\subseteq J$. Then letting $\widetilde{K}$ be the inverse
image of $K$ via the projection $G\r J$, $inf^J_G(X)=G/\widetilde{K}_+$. Now again by preservation of
homotopy cofibers, we may also assume that $Y=G/\Gamma_+$ for some closed subgroup $\Gamma\subseteq G$,
at which point \rref{elinfp} follows from the analogous consideration on spaces ($G$-orbits).
(We will show below in Subsection \ref{ssinfg} that for a $G$-universe $\mathscr{U}$ which is not $H$-fixed,
inflation does not have a left adjoint, so the left projection formula makes no sense.)

The right projection formula \rref{erproj} states that
$$\diagram X\wedge (Y^H)\rto^(.4)\sim & (inf_G^J(X)\wedge Y)^H,\enddiagram$$
which is also true by a similar induction. (In fact, the right projection formula
is true even without assuming that $\mathscr{U}$ is $H$-fixed, see in Subsection \ref{ssinfg} below.)

The ``dualizing object" in this case is the $G$-$\mathscr{U}$-suspension spectrum of $\widetilde{E[H]}$. The
functor $f^{(1)}$ does not generally preserve compact objects, since the dualizing object is not compact. 
(For example for $G=H=\Z/2$, and the trivial universe,
it suffices to show that $E\Z/2_+$ is not compact, which, since $?/H$ preserves
compacts, reduces to showing that $B\Z/2_+$ is not a compact spectrum. This is well known (and also
follows, for example, from Lemma \ref{lcounter} of the Appendix.) Thus, $f_{(-1)}$ in general does not
have a right adjoint.

On the other hand, $f_{(1)}$ does not in general preserve products. Again, for $G=H=\Z/2$ and the trivial
universe, consider the countable product of the spectra $E\Z/2_+\wedge K$ (where $K$ is, say, the fixed
$K$-theory $\Z/2$-spectrum). 
The countable product of these spectra still has trivial fixed points, so the canonical map 
$$E\Z/2_+\wedge 
\displaystyle\prod_\N K\r
\prod_\N (E\Z/2_+\wedge K)$$
is an equivalence. Applying $?/(\Z/2)$ and using \rref{elinfp}, we get 
$$\diagram
B\Z/2_+\wedge\displaystyle \prod_\N K\rto^(.4)\sim &(\displaystyle\prod_{\N}(E\Z/2_+\wedge K))/\Z/2\rto
&\displaystyle\prod_\N(B\Z/2_+\wedge K)
\enddiagram
$$
where the second map is the map which would be an equivalence if $f_{(1)}$ preserved
products, while the composition is \rref{ebcounter} of the Appendix,
which is not an equivalence by Lemma \ref{lcounter}. Thus, in general, $f_{(1)}$ does not have a left 
adjoint.

\vspace{3mm}

\subsection{Restriction} \label{ssres}
Let $f^*=res_H^G:DG$-$\mathscr{U}$-spectra
$\r DH$-$\mathscr{U}$-spectra be the forgetful functor where $\mathscr{U}$ is
any $G$-universe. We have a left adjoint
$f_{(1)}=G\ltimes_H?$ and a right
adjoint $f_*=F_H[G,?)$. It is well known \cite{lms} that the left projection 
formula \rref{elproj} holds. By Corollary \ref{cor1}, $f_*$ always has another right adjoint, $f^{(1)}$.
When $\mathscr{U}$ is the complete universe, then, of course, the assumptions
of \cite{sanders} are satisfied, and in fact the classical Wirthm\"uller isomorphism \cite{lms}
asserts that
we have an infinite chain of adjunctions. To get a feel for what the functor $f^{(1)}$ is like in general, let us
consider an example.

\vspace{3mm}

Let us consider the case when $G=\Z/2$, $H=\{e\}$, and $\mathscr{U}=
\R^\infty$ is the
trivial universe. In this case, we can construct a cofibration sequence for an $\{e\}$-spectrum $X$
\beg{ei1z2}{\Z/2_+\wedge X\r F(\Z/2_+,X)\r \widetilde{E\Z/2}\wedge inf_{\Z/2}^{\{e\}}X
}
as follows:
One has $f^*f_*(X)=X\times X$ since, non-equivariantly, $G$ is a $2$-point set,
and the first map \rref{ei1z2} is the adjoint to the canonical map $X\r X\times \{*\}\r X\times X$. The first
morphism \rref{ei1z2} is an equivalence non-equivariantly by stability (as it is the canonical morphism from
the coproduct to the product of two copies of $X$), but the source has trivial fixed points,
while the target has fixed points $X$ (embedded diagonally). Thus, the cofiber in \rref{ei1z2} has
fixed points $X$ and is trivial non-equivariantly. 

Now recall that the category $DG$-$\R^\infty$-spectra (``naive"
$G$-spectra) for $G$ finite is equivalent to the diagram derived category of functors from the
orbit category $\mathcal{O}_G$ into spectra. This ``folklore" fact about $G$-$\R^\infty$-spectra is proved
the same way as for $G$-spaces: For $G$-spaces, the forgetful functor $U$ from $\mathcal{O}_G^{Op}$-spaces
to $G$-spaces (by taking $X\mapsto X_{G/\{e\}}$) is left adjoint to the functor $\Psi$ where for a $G$-space $X$,
$\Psi(X)(G/H)=X^H$ (the right Kan extension). If we take object-wise equivalences on 
$\mathcal{O}_G^{Op}$-spaces,  and let $\Psi$
create equivalences on $G$-spaces, it is formal that the left derived functor $LU$ is an inverse to the total derived
functor $D\Psi$ on derived categories.

Now since the first map \rref{ei1z2} is an equivalence
non-equivariantly, and the first term has trivial fixed points, the homotopy cofiber of that map
is trivial non-equivariantly and has fixed points $X$. Thus, (see \cite{lms}, Section II.8),
it maps into $\widetilde{E\Z/2}\wedge inf_{\Z/2}^{\{e\}}X$, and the map induces an equivalence on
fixed points as well as non-equivariantly, and thus, is an equivalence.

Now \rref{ei1z2}
can be written in the framework \rref{ethree} as
\beg{ei1z2a}{f_{(1)}(X)\r f_*(X)\r \widetilde{E\Z/2}\wedge inf_{\Z/2}^{\{e\}}X.
}
By the results of Subsection \ref{inffixed}, 
the last term of the cofibration sequence \rref{ei1z2a} is in fact the right adjoint to the
functor $(?)^{\Z/2}$ on derived categories. 

\vspace{3mm} 
In fact, while it is not crucial for what follows, it is 
interesting to realize that for $X=S$, the cofibration \rref{ei1z2} can be realized geometrically by stabilizing the canonical cofibration sequence
\beg{essexx}{S^n\vee S^n \r S^n\times S^n \r S^n\wedge S^n,}
where
all terms are given the $\Z/2$-equivariant structure with the generator of $\Z/2$ switching factors. The last term
is equivariantly homeomorphic to $S^{n+n\alpha}$, where $\alpha$ is
the $1$-dimensional real sign representation. Desuspending $n$ times and taking a colimit with $n\r\infty$
(using the fact that $S^{\infty\alpha}$ is a model for $\widetilde{E\Z/2}$), the connecting map of
\rref{essexx}, after stabilization, gives a model of the connecting map of \rref{ei1z2} for $X=S$.

The connecting map of \rref{essexx} can then be described as a map
\beg{esp1}{S^{n+n\alpha}\r \Z/2_+\wedge S^{n+1}}
given as the (trivial) suspension of the following map: in $S^{2n-1}$,
consider an embedding of $S^{n-1}\times S^{n-1}$, thus splitting $S^{2n-1}$ into two solid tori $S^{n-1}\times D^n$.
Consider the $\Z/2$-equivariant structure where the generator swaps the two factors of $S^{n-1}\times S^{n-1}$,
and the two solid tori. Then we get a map into $S^{n}\vee S^{n}$ by collapsing $S^{n-1}\times S^{n-1}$
to a point, and projecting each solid torus $S^{n-1}\times D^n$ to $D^n$ (with the boundary collapsed to a point).
This map is equivariant when we consider on $S^n\vee S^n$ the $\Z/2$-equivariant structure where the generator
swaps factors. 

\vspace{3mm}
We can see that the connecting map is non-trivial by observing that applying $(?)/(\Z/2)$ to the
first morphism \rref{ei1z2}, we obtain the canonical map
\beg{ei1sp2}{X\r Sp^2(X)
}
which does not split for $X=S$ by considering Steenrod operations (cf. \cite{nakaoka,milgram}).

\vspace{3mm}
Let us study the cofibration sequence \rref{ei1z2} in more detail. The second morphism in \rref{ei1z2}
was constructed using obstruction theory. More explicitly, we have 
\beg{efuj}{F(\Z/2_+,X)^{\Z/2}=X=(\widetilde{E\Z/2}\wedge inf_{\Z/2}^{\{e\}}X)^{\Z/2},}
and $res^{\Z/2}_{\{e\}}
(\widetilde{E\Z/2}\wedge inf_{\Z/2}^{\{e\}}X)$ is contractible, and hence, for a cell $\{e\}$-spectrum 
$X$, the (non-equivariant) space $\mathscr{Q}_X$
of $\Z/2$-equivariant morphisms
$$F(\Z/2_+,X)\r \widetilde{E\Z/2}\wedge inf_{\Z/2}^{\{e\}}X$$
extending the morphism \rref{efuj} is contractible. Now for each $\alpha\in \mathscr{Q}_X$, the
composite morphism 
\beg{efuj1}{\diagram\Z/2_+\wedge X\rto^\iota & F(\Z/2_+,X)\rto^(.45)\alpha
& \widetilde{E\Z/2}\wedge inf_{\Z/2}^{\{e\}}X\enddiagram}
where $\iota$ is the first morphism of \rref{ei1z2}
is null-homotopic (since the source has trivial fixed points), and for the same reason, the space
$\mathscr{S}_X$ of pairs $(\alpha, h)$ where $h$ is a null-homotopy of \rref{efuj1}, is also contractible. 
Note that specifying an element of $\mathscr{S}_X$ is equivalent to specifying a morphism
of $\Z/2$-$\R^\infty$-spectra
\beg{efuj2}{C(\iota)\r \widetilde{E\Z/2}\wedge inf_{\Z/2}^{\{e\}}X}
extending the identification \rref{efuj} (where $C$ denotes the mapping cone). Thus, the space of such morphisms is
contractible. 
Further, every such morphism \rref{efuj2} is a $\Z/2$-equivariant equivalence. 

Similarly, one also sees that the
space of morphisms
\beg{efuj3}{\widetilde{E\Z/2}\wedge inf_{\Z/2}^{\{e\}}X\r C(\iota)
}
which are the identity on fixed points (using the identification \rref{efuj}) is contractible, and that each of those
morphisms is an equivalence. Even more generally, if we denote the morphism $\iota$ of \rref{efuj1} more
specifically by $\iota_X$, then for any morphism of $\{e\}$-spectra
\beg{efuj4}{f:X\r Y,
}
the space of all morphisms
$$\widetilde{E\Z/2}\wedge inf_{\Z/2}^{\{e\}}X\r C(\iota_Y)$$
which restrict to $f$ on fixed points is contractible. This implies, in particular, that the morphism \rref{efuj3}, and hence its 
composition
\beg{ei1z2b}{\widetilde{E\Z/2}\wedge inf_{\Z/2}^{\{e\}}X\r \Sigma\Z/2_+\wedge X}
with the canonical natural morphism
$$C(\iota)\r \Sigma \Z/2_+\wedge X$$
is natural in the derived category.
(See \cite{lms}, Section II.8 for further discussion of this method.) 

\vspace{3mm}
Now the natural transformation \rref{ei1z2b} in $D\Z/2$-$\R^\infty$-spectra has a right adjoint
\beg{ei1z2c}{ \Omega (res_{\{e\}}^{\Z/2}Z)\r F(\widetilde{E\Z/2},Z)^{\Z/2}
}
(here $Z$ is a $\Z/2$-$\R^\infty$-spectrum, and \rref{ei1z2c} is a morphism of $\{e\}$-spectra).
Continuing \rref{ei1z2b} to the right, we get a sequence of functors in $D\Z/2$-$\R^\infty$-spectra
$$\widetilde{E\Z/2}\wedge inf_{\Z/2}^{\{e\}}X\r \Sigma\Z/2_+\wedge X\r \Sigma F(\Z/2_+,X),$$
which is a cofibration sequence object-wise. Therefore, we have a right adjoint sequence
\beg{efuj10}{\Omega F(\widetilde{E\Z/2},Z)\r \Omega (res_{\{e\}}^{\Z/2}Z)\r F(\widetilde{E\Z/2},Z)^{\Z/2}}
where the composition of the two maps \rref{efuj10} is $0$, since the adjoint of the $0$ morphism is $0$.
Additionally, however, \rref{efuj10} induces a long exact sequence on homotopy groups by the adjunction, 
and thus is a cofibration sequence object-wise.

Using stability, we have a natural sequence of right adjoints on
the level of derived categories
\beg{ei1z2dual}{F(\widetilde{E\Z/2},Z)^{\Z/2}\r f^{(1)}Z\r f^{*}(Z).
}
where $Z$ is a $\Z/2$-$\R^\infty$-spectrum. It is worth noting that it is not obvious how to conclude this
directly on the level of triangulated categories: We may define a distinguished triangle of functors
$$A\r B\r C\r \Sigma A$$
as a sequence of natural transformations which is a distinguished triangle on every object. However, it is
then not obvious that if $A$ and one of the functors $B$, $C$ have a right or left adjoint, so does the third.
Additionally, even if this occurs, i.e. $A$, $B$, $C$ have right (or left) adjoints $A^\prime$, $B^\prime$, $C^\prime$,
it is then not obvious that the adjoint triangle
$$C^\prime\r B^\prime\r A^\prime\r \Sigma C^\prime$$
is distinguished. (This seems, in fact, like an interesting problem.)

This is the reason why a modern algebraic topologist seldom works fully in the triangulated category directly,
and always, implicitly or explicitly, 
has the underlying ``point set-level" (i.e. non-derived) category of spectra in mind. We shall see  
more examples of this technique below.

\vspace{3mm}
Note that from the cofibration sequence 
\rref{ei1z2dual} it follows that in this case $f^{(1)}$ does not have a right adjoint,
since $f_*$ does but $F(\widetilde{E\Z/2},?)^{\Z/2}$ does not, as already shown in Subsection \ref{inffixed}.
In effect, by Corollary \ref{cfive}, it suffices to observe that in \rref{ei1z2a}, $f_*$ does not preserve
compact objects, since $f_{(1)}$ preserves them and $\widetilde{E\Z/2}\wedge inf_{\Z/2}^{\{e\}}?$ 
does not (since $F(\widetilde{E\Z/2},?)^{\Z/2}$ has no derived right adjoint).

\vspace{3mm}
However, using the same method,
from formula \rref{ei1z2}, we see that inductively, $\Z/2_+\wedge ?$ has as many left adjoints as
the last term of the cofibration $\widetilde{E\Z/2}\wedge inf_{\Z/2}^{\{e\}}?$. They key point is that
the homotopy fiber $F$ functor from the category of morphisms of (non-derived) spectra
has a left adjoint, namely the homotopy
cofiber (i.e. mapping cone) functor $C$. Additionally, for a morphism $f:X\r Y$ in the non-derived category of spectra,
if we denote by $i:Y\r Cf$ the canonical natural morphism, then we have a canonical natural morphism
$X\r Fi$, which is an equivalence  if $X$, $Y$ are cell. Similar comments apply to
$G$-$\mathscr{U}$-spectra for any $G$, $\mathscr{U}$. Thus, starting from a homotopy cofiber
sequence of functors on the non-derived level, we may replace it by a homotopy fiber sequence on the 
non-derived level, which has a left adjoint, which is a homotopy cofiber sequence. Additionally, if the functors
in question preserve homotopy, we get a corresponding adjunction on the derived level. Provided
two of the left adjoint functors have again left adjoints on the non-derived level, we can iterate this procedure, starting
at every stage on the non-derived level, and noting we have a corresponding adjunction on the left derived level.
Applying this to \rref{ei1z2}, when $\widetilde{E\Z/2}\wedge inf_{\Z/2}^{\{e\}}?$ has
no further iterated left adjoints, by Corollary \ref{cfive}, it does not preserve products, and hence the corresponding
iterated left adjoint of $\Z/2_+\wedge ?$ does not, as the remaining term of the (co)fiber sequence does. (More
details could be given here, but below we will actually compute these functors explicitly.)
As we showed in
Subsection \ref{inffixed}, the left derived functor of $\widetilde{E\Z/2}\wedge inf_{\Z/2}^{\{e\}}?$
has a sequence of a total of $3$ left adjoints, and not more. 

Thus, in this case, precisely
the functors $f_{(-1)}$, $f_{(2)}$, $f^{(-2)}$ exist, leading to a chain of a total of $7$ adjoints. 
In fact, these functors can be described geometrically by taking successive left adjoints of \rref{ei1z2}:
We have
$$f^{(-1)}X= res_{\{e\}}^{\Z/2}X/X^{\Z/2},$$
$f_{(2)}Y$ is the fiber of the canonical morphism $\Z/2_+\wedge Y \r inf_{\Z/2}^{\{e\}} Y$, and
$f^{(-2)}X$ is the cofiber of a morphism 
$$X/(\Z/2)\r res_{\{e\}}^{\Z/2}X/X^{\Z/2},$$
which is a variant of the transfer. Again, we see that this functor has no left adjoint.

\vspace{3mm}
One can also see that the right projection formula \rref{erproj} fails for the case
$G=\Z/2$, $H=\{e\}$ and the trivial universe. In effect, if this formula were true, it would say (putting $Y=S$)
that for a
$\Z/2$-$\R^\infty$-spectrum $X$, the canonical morphism
\beg{ei1rp}{X\wedge F(\Z/2_+,S)\r F(\Z/2_+,inf_{\Z/2}^{\{e\}}res_{\{e\}}^{\Z/2}X)
}
is an equivalence. Since the analogous statement with $F(\Z/2_+,?)$ replaced by $\Z/2_+\wedge ?$ holds, 
this is equivalent to the canonical morphism
\beg{ei1rp1}{\widetilde{E\Z/2}\wedge X\r \widetilde{E\Z/2}\wedge inf_{\Z/2}^{\{e\}}res_{\{e\}}^{\Z/2}X
}
being an equivalence (since $\widetilde{E\Z/2}\wedge X\sim \widetilde{E\Z/2}\wedge (X^{\Z/2})_{fixed}$, the 
canonical morphism \rref{ei1rp1} is induced by the canonical morphism $X^{\Z/2}\r E_{\{e\}}$, which
is also \rref{ei1rp1} on fixed points). Thus, \rref{ei1rp1} is not an equivalence when $X$ is not fixed, 
and hence neither is \rref{ei1rp}.

\vspace{3mm}
What in this example can be generalized? Let us specialize to the case of a finite group $G$ and the trivial
universe. (The main significance of the finiteness being that the orbit category is finite.) In this case, the 
cofiber sequence \rref{ei1z2} generalizes to 
\beg{ei1g}{G\ltimes_H X\r F_H[G,X)\r \widetilde{E\mathscr{F}(H)}\wedge F_H[G,X)
}
where $\mathscr{F}(H)$ is the family of all subgroups of $G$ subconjugate to $H$. The cofibration sequence \rref{ei1g} 
can be used to gain information on both left and right adjoints. Recall that for a subgroup $K$ of $G$ which
is not subconjugate to $H$, we have 
$$F_H[G,X)^K\sim \prod_{a\in K\backslash G/H}X^{H\cap a^{-1}Ka}.$$
Denote, as usual, by $N(K)$ the normalizer of $K$, and $W(K)=N(K)/K$. Recall that, for a general (not 
necessarily normal) subgroup $K$ of $G$, $(?)^K$ is a functor $DG$-$\R^\infty$-spectra
$\r DW(K)$-$\R^\infty$-spectra.
As a $W(K)$-spectrum,
\beg{efgx}{F_H[G,X)^K\sim \bigvee_{[a]\in N(K)\backslash G/H} F_{W(K,H)}[W(K), X^{H\cap a^{-1}Ka})}
where 
$$W(K,H)=\frac{a^{-1}Ka\cdot H\cap a^{-1}N(K)a}{a^{-1}Ka}.$$
We can express $\widetilde{E\mathscr{F}(H)}\wedge F_H[G,X)$ as a finite homotopy (co)limit of 
$G$-$\R^\infty$-spectra of the form
$$G\ltimes_{N(K)}(F_H[G,X)^K)$$
where $K$ is not subconjugate to $H$, so by formula \rref{efgx} and \rref{ei1g}, we can, in principle, inductively 
write down a model for the right adjoint of $F_H[G,X)$ (since $|W(K)|<|G|$).

\vspace{3mm}
On the other hand, formula \rref{ei1g} can also be used to construct two left adjoints to $G\ltimes_HX$. By induction,
again, it suffices to prove that the functor $\widetilde{E}\mathscr{F}[H]\;\wedge ?$ has two left adjoints. This follows
from the following general fact, which can be traced back to \cite{lms}.

\begin{proposition}\label{plms}
For any family $\mathscr{F}$ of subgroups of a finite group $G$, the functor 
\beg{eplms}{\widetilde{E\mathscr{F}}\wedge ?:DG\text{-}\mathscr{\R^\infty}\text{-spectra}\r 
DG\text{-}\mathscr{\R^\infty}\text{-spectra}
}
has two left adjoints.
\end{proposition}

\begin{proof}
Recall that the category $DG\text{-}\mathscr{\R^\infty}\text{-spectra}$ is equivalent to the 
diagram derived category of $\mathcal{O}_G^{Op}$-spectra (i.e. with object-wise
equivalences) where $\mathcal{O}_G$ is the orbit 
category of $G$. We can identify a family $\mathscr{F}$ with a full subcategory of $\mathcal{O}_G$
on the subgroups which belong to the family, and the corresponding {\em co-family} $\widetilde{\mathscr{F}}$
with the full subcategory of $\mathcal{O}_G$ on all the remaining subgroups. 

For an $\mathscr{F}^{Op}$-spectrum $X$, a $G$-$\R^\infty$-spectrum $\widetilde{E\mathscr{F}}\wedge X$ is
well-defined, and if we denote this functor by $\phi_*$, it is right adjoint to the forgetful
functor $\phi^*$ from $G$-$\R^\infty$-spectra (=$\mathcal{O}_G^{Op}$-spectra) to 
$\widetilde{\mathscr{F}}^{Op}$-spectra. This functor then has a left adjoint $\phi_\sharp=\mathcal{O}_G^{Op}\ltimes_{
\widetilde{\mathscr{F}}}?$. The left adjoint to \rref{eplms} is $\phi_\sharp \phi^*$.

To show that a further left adjoint exists, by Corollary \ref{cfive}, it suffices to show that the functor
$\phi_\sharp \phi^*$ preserves products. This follows from the following result.
\end{proof}

\begin{lemma}\label{lproducts}
Let $\mathscr{C}$ be a finite category in which every endomorphism has an inverse, and let $F$ be a 
contravariant functor from $\mathscr{C}$ to finite sets. Suppose that for every $x\in Obj(\mathscr{C})$,
$\mathscr{C}(x,x)$ acts freely (from the right) on $F(x)$. Then the functor
\beg{ebarc}{B_\wedge(F_+,\mathscr{C}_+,?):\mathscr{C}\text{-spectra}\r\text{Spectra}}
preserves products.
\end{lemma}

\begin{proof}
In the case when $\mathscr{C}$ is a group, under our assumptions, the functor \rref{ebarc} is just a 
``sum of finitely many copies", so it preserves products. In the general case, it arises from the corresponding
functors for the automorphism groups of $\mathscr{C}$ by a finite homotopy colimit over a poset, so the
conclusion follows from stability.
\end{proof}

\vspace{3mm}
It is not difficult to give an example of an inclusion of finite group $H\subset G$ where the functor
$G\ltimes_H?$ from $H$-$\R^\infty$-spectra to $G$-$\R^\infty$-spectra does not have three left adjoints. Let 
$G=\Z/4$, $H=\Z/2$. In this case, \rref{ei1g} becomes
\beg{ebarc1}{(\Z/4)\ltimes_{\Z/2}X\r F_{\Z/2}[\Z/4,X)\r \widetilde{E[\Z/4]}\wedge X^{\Z/2},}
so again by induction, the existence of three left adjoints to $(\Z/4)\ltimes_{\Z/2}?$ would be equivalent
to the existence of three left adjoints to $\widetilde{E[\Z/4]}\wedge X^{\Z/2}$. We know however
that the left adjoint to that functor is $inf_{\Z/2}^{\{e\}}(?)^{\Z/4}$, whose left adjoint, in turn, is
$inf_{\Z/4}^{\{e\}}(?/(\Z/2)$. This functor does not preserve products, since $?/(\Z/2)$ does not,
and $inf_{\Z/4}^{\{e\}}$ has a left inverse $?^{\Z/4}$, which preserves products. Thus, 
$inf_{\Z/4}^{\{e\}}(?/(\Z/2))$ has no left adjoint, as claimed.

\vspace{3mm}
On the other hand, if $G$ is finite abelian, the functor $G\ltimes ?$ from spectra to $G$-$\R^\infty$-spectra
does have three left adjoints (and hence $res_{\{e\}}^G$ has four left adjoints, leading to a chain of
7 adjoints). Again, it suffices to show that the rightmost term of the cofibration sequence \rref{ei1g}, which
in this case is $\widetilde{EG}\wedge F[G,?)$, has three left adjoints. Again, we can model $\widetilde{EG}$
as a finite homotopy colimit of $\widetilde{E[H]}$ for subgroups $\{e\}\neq H\subseteq G$, and therefore
it suffices to show that $\widetilde{E[H]}\wedge F[G,?)\sim \widetilde{E[H]}\wedge F[G,?)^H$ have three
left adjoints. However, for a spectrum $X$, $F[G,X)^H$ is a sum of (finitely many) copies of $X$, so we are
reduced to showing that $\widetilde{E[H]}\wedge X$ from spectra to $G$-$\R^\infty$-spectra
has three left adjoints. But the left adjoint to that functor is $(res_H^G?)^H$, which we already know has
two left adjoints.

\vspace{3mm}
\subsection{Smashing with a finite spectrum} \label{sssmash}
Here we are not dealing with symmetric monoidal functors, so there is no discussion
of projection formulas. The functor $X\wedge ?$ has a right adjoint $F(X,?)$ for any spectrum $X$,
which passes on to derived categories. Additionally, by Corollary \ref{cfive}, on derived
categories, $F(X,?)$ has an additional right adjoint if and only if $X\wedge ?$ preserves compact
objects, which happens if and only if $X$ is itself compact. If $\mathscr{U}$ is a complete universe, we of course have $F(X,?)=DX\wedge ?$, so there is an infinite chain
of adjunctions in both directions. 

Note that if $X$ is of the form $G/H_+$
where $H$ is a closed subgroup of $G$, on the level of derived categories, $X\wedge ?$ is isomorphic
to $G\ltimes (res_H^G?)$, so we already
know that it has two right adjoints. In the case when $X$ is finite, therefore,
the cellular filtration on $X$ gives a filtration on the second right adjoint, where the associated graded
pieces can, in principle, be described by the methods of Subsection \ref{ssres}.

\vspace{3mm}
We can show that the right adjoint to the endofunctor $F(\Z/2_+,?)$ in
$D\Z/2$-$\R^\infty$-spectra has no right adjoint. 
In effect, considering Corollary \ref{cfive}, if $F(\Z/2_+,?)$ preserved compact objects, 
$F(\Z/2_+,inf_{\Z/2}^{\{e\}}(?))$ (i.e. the middle term of \rref{ei1z2}) would. So, since
the first term of \rref{ei1z2} preserves compactness, it would follow that 
$\widetilde{E\Z/2}\wedge inf_{\Z/2}^{\{e\}}?$ preserves compactness, which we already
showed is not the case.

\vspace{3mm}
Regarding left adjoints, it follows from what we showed in Subsection \rref{ssres} that the endofunctor
$G\ltimes (res_H^G?)$, and hence the endofunctor $X\wedge ?$ for $X$ a finite spectrum, in $DG$-$\R^\infty$-spectra
with $G$ finite has two left adjoints (again, we consider the functors on the non-derived level,
and take strictly functorial homotopy (co)fibers, working inductively on the 
number of cells of $X$), thus leading to a chain of $5$ adjoints involving $X\wedge ?$. 
Furthermore, for $G=\Z/p$, since the only orbits are trivial and $\Z/p$, also by the results of Subsection
\ref{ssres}, we have a third left adjoint, leading to a chain of $6$ adjoints. 

\vspace{3mm}
In the case when $G=\Z/2$, $\mathscr{U}=\R^\infty$, we can show that no further left adjoints exist.
In effect, by the results of Subsection \rref{ssres}, the first left adjoint to $\Z/2\ltimes res_{\{e\}}^{\Z/2}X$
is $\Z/2\ltimes (res_{\{e\}}^{\Z/2}X/X^{\Z/2})$, and hence the second left adjoint is the fiber of 
the canonical morphism
\beg{exsecond}{
\Z/2\ltimes (res_{\{e\}}^{\Z/2}X/X^{\Z/2}))\r inf_{\Z/2}^{\{e\}}(res_{\{e\}}^{\Z/2}X/X^{\Z/2}).}
Since the first one of these functors has two left adjoints, it suffices to show that the second one does not. 
Now the left adjoint to the second functor \rref{exsecond} is the cofiber of
$$\Z/2\ltimes (X/(\Z/2))\r inf_{\Z/2}^{\{e\}}(X/(\Z/2))$$
which does not preserve products (for example, applying $res_{\{e\}}^{\Z/2}$, we get $X/(\Z/2)$ again,
which we already showed does not preserve products. Thus, we have a chain of precisely $6$ adjoints
in this case.

\vspace{3mm}
In the case when $G=\Z/4$, $\mathscr{U}=\R^\infty$, $X=(\Z/4)/(\Z/2)_+$, we can show that a third left
adjoint does not exist, thus leading to a chain of precisely $5$ adjoints. In effect, by \rref{ebarc1},
it suffices, again, to work with the  endofunctor $\widetilde{E[\Z/4]}\wedge (res_{\Z/2}^{\Z/4}X)^{\Z/2}$
instead. Its left adjoint is $\Z/4\ltimes_{\Z/2} inf_{\Z/2}^{\{e\}}(X^{\Z/4})$, which we can represent as the
fiber of the morphism
$$F_{\Z/2}[\Z/4,inf_{\Z/2}^{\{e\}}(X^{\Z/4}))\r \widetilde{E[\Z/4]}\wedge (inf_{\Z/2}^{\{e\}}(X^{\Z/4}))^{\Z/2}.
$$
The second functor is $\widetilde{E[\Z/4]}\wedge X^{\Z/4}$, which we already know has two
left adjoints, so it suffices to show that $F_{\Z/2}[\Z/4,inf_{\Z/2}^{\{e\}}(X^{\Z/4}))$
does not. In effect, its left adjoint is
$inf_{\Z/4}^{\{e\}}((res_{\Z/2}^{\Z/4}X)/(\Z/2))$. This functor does not preserve products: Since
$inf_{\Z/4}^{\{e\}}$ has a left inverse $res_{\{e\}}^{\Z/4}$ which preserves products, it suffices
to show that $(res_{\Z/2}^{\Z/4}X)/(\Z/2)$ does not preserve products, but that follows from the same
example as before.

\vspace{3mm}
\subsection{Change of universe} Let $i:\mathscr{U}\r\mathscr{V}$ be an isometry of $G$-universes. Then recall
(\cite{lms}, Section II.1) that
a $G$-$\mathscr{V}$-spectrum $Y$ as an object of the derived category can be described by describing the
$G$-$\mathscr{U}$-spectra $i^*\Sigma^V Y$ where $V$ runs through finite-dimensional
$G$-representations contained in 
$\mathscr{V}$ (and it suffices to consider those representations $V$ which do not have irreducible summands in 
$\mathscr{U}$). Thus, to describe $i_*$, it suffices to describe $i^*\Sigma^V i_*$, which is right adjoint to
\beg{ei3la}{i^*\Sigma^{-V}i_\sharp X=\Omega^Vi^*i_\sharp X=
\operatornamewithlimits{hocolim}_{W}\Omega^{W+V}\Sigma^WX}
where again, $W$ runs through subrepresentations of $\mathscr{V}$ with no irreducible summand contained in 
$\mathscr{U}$.
But we know that $\Omega^W$ has right adjoint $R^W=R(S^W,?)$, so the right adjoint
to \rref{ei3la} is
\beg{ei3ea}{i^*\Sigma^Vi_* X=\operatornamewithlimits{holim}_{W} \Omega^WR^{W+V}X.}
From this point of view, we have a description of the functor $i_*$.

\vspace{3mm}
Regarding additional adjoints, in general, $i_\sharp$ does not preserve products, and thus does not have a 
left adjoint. To see this, let us consider again the case $G=\Z/2$, where $\mathscr{U}=\R^\infty$ is
the trivial universe and $\mathscr{V}$ is the complete universe. We will show that the functor 

\begin{center}
$(i^*(i_\sharp(inf_{\Z/2}^{\{e\}})))^{\Z/2}:D$Spectra$\r D$Spectra
\end{center}

\noindent
does not preserve products, which is sufficient, since the functors $(?)^{\Z/2}$, $i^*$, $inf_{\Z/2}^{\{e\}}$
do preserve products. In effect, it is well known that we have a cofibration sequence
$$B\Z/2_+\wedge X\r (i^*(i_\sharp(inf_{\Z/2}^{\{e\}}X)))^{\Z/2}\r X$$
(see for example \cite{lms}, Section II.7, or, even much more explicitly, and generally, Section V.11)
where the third term preserves products, and again, the connecting map is
natural (it is, in fact, in this case $0$), so it suffices to prove that the first term does not, which is Lemma
\ref{lcounter}.

\vspace{3mm}
For the same reason, $i^*$ does not preserve compact objects (since the functors
$(?)^{\Z/2}$, $inf_{\Z/2}^{\{e\}}$ do), so $i_*$ does not have a right adjoint. 

\vspace{3mm}
The right projection formula \rref{erproj}, for $x=\Z/2_+$, $y=S$, (where $G=\Z/2$, $\mathscr{U}$ is
the trivial universe and $\mathscr{V}$ is the complete universe) would say that
$$\Z/2_+\wedge i^*(S)\sim i^*(\Z/2_+).$$
We know that the right hand side has fixed points by the Wirthm\"{u}ller isomorphism, while the left hand side
does not. Therefore, \rref{erproj} is false.

\vspace{3mm}
\subsection{Inflation - the general case.} \label{ssinfg}
By formula \rref{einfch}, the general case of the inflation reduces
to the case of an $H$-fixed universe, and change of universes. However, it remains to resolve the question
of how many adjoints we have, and the question of a projection formula. 

In the case when $G=\Z/2$ and $\mathscr{U}$ is the complete universe, we have, again, for a spectrum $X$,
a cofibration 
sequence
$$B\Z/2_+\wedge X\r (inf_{\Z/2}^{\{e\}}X)^{\Z/2}\r X,$$
thus showing that $(?)^{\Z/2}$ does not preserve compact objects, since $inf_{\Z/2}^{\{e\}}$ does. Therefore,
the right adjoint to $(?)^{\Z/2}$ in this case does not have a further right adjoint by Corollary \ref{cfive}.
On the other hand, this situation actually satisfies the assumptions of \cite{sanders}, and therefore, we 
also know that the
functor $inf_{\Z/2}^{\{e\}}$ does not have a left adjoint.

\vspace{3mm}
We also know from \cite{sanders} that in the case of inflation from a complete universe to a complete universe,
the right projection formula \rref{erproj} is satisfied. However, it turns out to be true in general, which is curious, since
it is false for change of universe. In effect, in the general case, the right projection formula asserts an equivalence
\beg{eddd}{\diagram
X\wedge (Y^H)\rto^(.4)\sim & ((inf_G^JX)\wedge Y)^H.
\enddiagram
}
Since both sides are stable under desuspensions by finite subrepresentations of $\mathscr{U}^H$, it suffices
to consider the case when $X$ is a $J$-space. In that case, however, (when applied to a cell spectrum $Y$), the 
$V$'th space of both sides for $V$ $H$-fixed is the colimit of $\Omega^W(X\wedge Y_{V+W})$ over finite 
subrepresentations $W$ of $\mathscr{U}$.

\vspace{3mm}
\subsection{Geometric fixed points.} As already remarked, the most universal case is the case of
a complete universe. In that case, the assumptions of \cite{sanders} are satisfied, so we know that there
are two right adjoints, and the right projection formula \rref{erproj} holds. This turns out to be the case
in general. In fact, the geometric fixed point functor coincides with the fixed point functor in the case of
an $H$-fixed universe, so this case also generalizes the rightmost three functors of the chain of
$5$ adjoints for inflation in the case when $\mathscr{U}$ is $H$-fixed. By the same arguments, then, 
one shows that in general, the right adjoint to $\Phi^H$ is 
\beg{ephir}{\widetilde{E[H]}\wedge inf_G^J(?),} 
which, in turn, has the right adjoint $(F(\widetilde{E[H]},?)^H$. It is easy to see that in the case
of a complete universe with $G=\Z/2$, $H=\{e\}$, \rref{ephir} does not preserve compact objects,
since $\widetilde{E\Z/2}$ is not compact (as $E\Z/2_+$ is not). Therefore, the second right adjoint to
$\Phi^{\Z/2}$ does not have an additional right adjoint in this case, and by \cite{sanders}, $\Phi^{\Z/2}$ does
not have a left adjoint. Thus, in this case, we have the ``$3$-scenario" of \cite{sanders}.

\vspace{3mm}
In the case of a complete $G$-universe $\mathcal{U}$, we have the right projection formula
by \cite{sanders}, but in fact, again, it is true in general: it asserts an equivalence (for cell spectra)
of the form
$$\diagram
X\wedge \widetilde{E[H]}\wedge inf_J^GY\rto^\sim & \Phi^HX\wedge \widetilde{E[H]}\wedge Y,
\enddiagram
$$
which holds for the same reason as in the case of an $H$-fixed universe.

\section{Appendix}

We record here some auxilliary results, which we consider known, but for which we could not find
an easy reference.

\begin{lemma}\label{lcounter}
Let $K$ denote (non-equivariant complex) periodic $K$-theory. Then the canonical morphism
\beg{ebcounter}{\diagram
\displaystyle B\Z/2_+\wedge (\prod_\N K)\rto^\nsim &\displaystyle\prod_\N (B\Z/2_+\wedge K)
\enddiagram}
is not an equivalence.
\end{lemma}

\begin{proof}
We are trying to show that in a particular case, Borel homology does not
preserve products. Since Borel cohomology preserves products, it suffices to work with Tate cohomology 
instead (recall the bottom row of the Tate diagram, which gives a cofibration sequence
between Borel homology, Borel cohomology and Tate cohomology; for
background on this, see \cite{gm}). Now the Borel cohomology of $K$ has coefficients (in dimension $0$) $\Z\oplus \Z_2$ where the Euler class
maps the first summand to the second, so the Tate cohomology is $\Q_2$. Similarly, the Tate cohomology
of a countable product of copies of $K$ is $2^{-1}(\displaystyle \prod_\N \Z_2)$, the canonical map of which 
into $\displaystyle \prod_\N \Q_2$ is not an isomorphism (because of non-uniformity of denominators).
\end{proof}

\begin{lemma}\label{lsmall}
Let $Y$ be a $T_1$-space and suppose we have an indexing set $I$, and for each $F\subset\subset I$ (meaning a
finite subset) a subspace $Y_F\subseteq Y$ (with the induced topology) such that
\begin{enumerate}
\item\label{sim1}
$\displaystyle Y=\bigcup_{F\subset\subset I} Y_F$ (with the colimit topology)

\item\label{sim2}
$Y_F\cap  Y_G=Y_{F\cap G}$ for $F,G\subset\subset I$

\item\label{sim3}
$F\subseteq G\Rightarrow Y_F\subseteq Y_G$.

\end{enumerate}
Suppose $f:X\r Y$ is a continuous map where $X$ is compact. Then there exists $F\subset\subset I$ such that
$f(X)\subseteq Y_F$.
\end{lemma}

\begin{proof}
If there exists no $F\subset\subset I$ with $F(X)\subseteq Y_F$, then, by induction, we can construct sets 
$$\emptyset=F_0\subset F_1\subset F_2\subset\dots \subset\subset I$$
and points
$$y_n\in (Y_{F_n}\cap f(X))\smallsetminus Y_{F_{n-1}}$$
for $n=1,2,\dots$. (Once we have constructed $F_{n-1}$, our assumption $f(X)\nsubseteq Y_{F_n-1}$
implies there is a point $y_n\in f(X)\smallsetminus Y_{F_{n-1}}$. By (1), $y_n\in Y_F$ for some
$F\subset\subset I$, and we can take $F_n:=F_{n-1}\cup F$.)

We claim that
\beg{esmall4}{
\parbox{3.5in}{If $G\subset\subset I$ and $S$ is an infinite subset of $\{0,1,2,\dots\}$, then
$\{y_s\mid s\in S\}\nsubseteq Y_G$.}
}
To see this, select such a $G$ and $S$. Then there is a natural number $n$ such that
$$G\cap (\bigcup_n F_n)=G\cap F_n.$$
Assuming that $y_{n+1}\in Y_G$, we would have, by assumption \rref{sim2},
and assumption \rref{sim3},
$$y_{n+1}\in Y_G\cap Y_{F_{n+1}}=Y_{G\cap F_{n+1}}=Y_{G\cap F_n}\subseteq Y_{F_n},$$
which is a contradiction. 

It follows from \rref{esmall4} and assumption (1) that for every $n$, the set
$$T_n=\{y_n,y_{n+1},\dots\}\subseteq Y$$
is closed in $Y$. (Indeed, for any $F\subset\subset I$, $T_n\cap Y_F$ must be a finite set.) Hence
$$f^{-1}(T_1)\supseteq f^{-1}(T_2)\supseteq\dots$$
are non-empty closed subsets of $X$, whereas
$$\bigcap f^{-1}(T_n)=f^{-1}(\bigcap T_n)= f^{-1}(\emptyset)=\emptyset,$$
contradicting the compactness of $X$.
\end{proof}

\begin{corollary}\label{ccompact}
Shift suspensions and desuspensions of 
suspension spectra of orbits are compact objects in the category $DG$-$\mathscr{U}$-spectra
for any compact Lie group $G$ and any universe $\mathscr{U}$.
\end{corollary}

\begin{proof}
Consider $G$-$\mathscr{U}$-cell spectra $Z_i$, $i\in I$. 
The assumptions of Lemma \ref{lsmall} are satisfied first for a wedge of based spaces, and hence are also
satisfied for
$$Y=\operatornamewithlimits{colim}_W \Omega^{W} (\bigvee_{i\in I} (Z_i)_{V+W}),$$
$$Y_F=\operatornamewithlimits{colim}_W \Omega^{W} (\bigvee_{i\in F} (Z_i)_{V+W}).$$
Now recall from \cite{lms}, Sections I.2, I.3 that the spaces $Y$, $Y_F$ are the constituent spaces of 
coproducts of the spectra $Z_i$ over $I$, $F$. Additionally, maps from shift suspensions and desuspensions
of suspension spectra of orbits into a spectrum are, by adjunction, the same thing as maps from suspensions
of orbits into the constituent spaces, at which point we can apply Lemma \ref{lsmall}.
\end{proof}

\end{document}